\documentclass[11pt]{article}
\usepackage{e-jc}
\usepackage{amsthm,amsmath,amssymb}

\usepackage[colorlinks=true,citecolor=blue,linkcolor=black,urlcolor=blue]{hyperref}

\usepackage{graphicx}

\numberwithin{equation}{section}
\numberwithin{figure}{section}
\theoremstyle{plain}
\newtheorem{thm}{Theorem}[section]
\newtheorem{lem}[thm]{Lemma}

\theoremstyle{remark}

\newcommand{\M}{\operatorname{M}}

\title{A Generalization of Aztec Dragons}
\author{Tri Lai\footnote{This research was supported in part by the Institute for Mathematics and its Applications with funds provided by the National Science Foundation (grant no. DMS-0931945).}\\
\small Institute for Mathematics and its Applications\\[-0.8ex]
\small University of Minnesota\\[-0.8ex]
\small Minneapolis, MN 55455\\
\small email: \texttt{tmlai@ima.umn.edu}\\
\small website: \url{http://www.ima.umn.edu/~tmlai/}
}

\date{\small Mathematics Subject Classifications: 05A15,  05B45}

\begin{document}
\maketitle

\begin{abstract}
Aztec dragons are lattice regions first introduced by James Propp, which have the number of tilings given by a power of $2$. This family of regions has been investigated further by a number of authors. In this paper, we consider a generalization of the
 Aztec dragons to two new families of $6$-sided regions. By using Kuo's graphical condensation method, we prove that the tilings of the new regions are always enumerated by powers of $2$ and $3$.

\end{abstract}

\section{Introduction}
A lattice partitions the plane into \emph{fundamental regions}. A \textit{region} considered in this paper is a finite connected union of fundamental regions. We call the union of any two fundamental regions sharing an edge a \textit{tile}. We would like to know how many different ways to cover a certain region by tiles so that there are no gaps or overlaps; and such coverings are called \textit{tilings}. We use the notation $\M(R)$ for the number of tilings of a region $R$.

\begin{figure}\centering
\setlength{\unitlength}{4144sp}%
\begingroup\makeatletter\ifx\SetFigFont\undefined%
\gdef\SetFigFont#1#2#3#4#5{%
  \reset@font\fontsize{#1}{#2pt}%
  \fontfamily{#3}\fontseries{#4}\fontshape{#5}%
  \selectfont}%
\fi\endgroup%
\resizebox{12cm}{!}{
\begin{picture}(0,0)%
\includegraphics{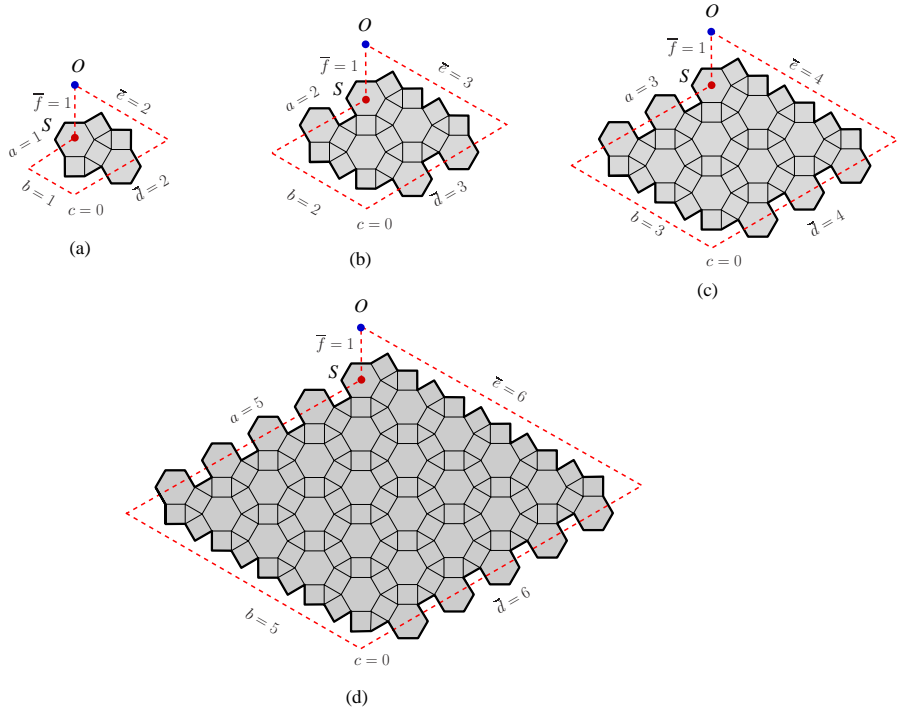}%
\end{picture}%
%
%

\begin{picture}(17894,14214)(1398,-13528)
\put(6031,-7834){\rotatebox{30.0}{\makebox(0,0)[lb]{\smash{{\SetFigFont{20}{24.0}{\rmdefault}{\mddefault}{\itdefault}{$a=5$}%
}}}}}
\put(8453,-12634){\makebox(0,0)[lb]{\smash{{\SetFigFont{20}{24.0}{\rmdefault}{\mddefault}{\itdefault}{$c=0$}%
}}}}
\put(11341,-11611){\rotatebox{30.0}{\makebox(0,0)[lb]{\smash{{\SetFigFont{20}{24.0}{\rmdefault}{\mddefault}{\itdefault}{$\overline{d}=6$}%
}}}}}
\put(6271,-11671){\rotatebox{330.0}{\makebox(0,0)[lb]{\smash{{\SetFigFont{20}{24.0}{\rmdefault}{\mddefault}{\itdefault}{$b=5$}%
}}}}}
\put(11183,-7054){\rotatebox{330.0}{\makebox(0,0)[lb]{\smash{{\SetFigFont{20}{24.0}{\rmdefault}{\mddefault}{\itdefault}{$\overline{e}=6$}%
}}}}}
\put(7696,-6346){\makebox(0,0)[lb]{\smash{{\SetFigFont{20}{24.0}{\rmdefault}{\mddefault}{\itdefault}{$\overline{f}=1$}%
}}}}
\put(2071,-1546){\makebox(0,0)[lb]{\smash{{\SetFigFont{20}{24.0}{\rmdefault}{\mddefault}{\itdefault}{$\overline{f}=1$}%
}}}}
\put(7771,-811){\makebox(0,0)[lb]{\smash{{\SetFigFont{20}{24.0}{\rmdefault}{\mddefault}{\itdefault}{$\overline{f}=1$}%
}}}}
\put(14686,-451){\makebox(0,0)[lb]{\smash{{\SetFigFont{20}{24.0}{\rmdefault}{\mddefault}{\itdefault}{$\overline{f}=1$}%
}}}}
\put(13881,-1446){\rotatebox{30.0}{\makebox(0,0)[lb]{\smash{{\SetFigFont{20}{24.0}{\rmdefault}{\mddefault}{\itdefault}{$a=3$}%
}}}}}
\put(7141,-1636){\rotatebox{30.0}{\makebox(0,0)[lb]{\smash{{\SetFigFont{20}{24.0}{\rmdefault}{\mddefault}{\itdefault}{$a=2$}%
}}}}}
\put(1621,-2536){\rotatebox{30.0}{\makebox(0,0)[lb]{\smash{{\SetFigFont{20}{24.0}{\rmdefault}{\mddefault}{\itdefault}{$a=1$}%
}}}}}
\put(1831,-3166){\rotatebox{330.0}{\makebox(0,0)[lb]{\smash{{\SetFigFont{20}{24.0}{\rmdefault}{\mddefault}{\itdefault}{$b=1$}%
}}}}}
\put(7096,-3331){\rotatebox{330.0}{\makebox(0,0)[lb]{\smash{{\SetFigFont{20}{24.0}{\rmdefault}{\mddefault}{\itdefault}{$b=2$}%
}}}}}
\put(13936,-3766){\rotatebox{330.0}{\makebox(0,0)[lb]{\smash{{\SetFigFont{20}{24.0}{\rmdefault}{\mddefault}{\itdefault}{$b=3$}%
}}}}}
\put(17581,-4111){\rotatebox{30.0}{\makebox(0,0)[lb]{\smash{{\SetFigFont{20}{24.0}{\rmdefault}{\mddefault}{\itdefault}{$\overline{d}=4$}%
}}}}}
\put(10081,-3526){\rotatebox{30.0}{\makebox(0,0)[lb]{\smash{{\SetFigFont{20}{24.0}{\rmdefault}{\mddefault}{\itdefault}{$\overline{d}=3$}%
}}}}}
\put(4152,-3413){\rotatebox{30.0}{\makebox(0,0)[lb]{\smash{{\SetFigFont{20}{24.0}{\rmdefault}{\mddefault}{\itdefault}{$\overline{d}=2$}%
}}}}}
\put(3706,-1328){\rotatebox{330.0}{\makebox(0,0)[lb]{\smash{{\SetFigFont{20}{24.0}{\rmdefault}{\mddefault}{\itdefault}{$\overline{e}=2$}%
}}}}}
\put(10111,-796){\rotatebox{330.0}{\makebox(0,0)[lb]{\smash{{\SetFigFont{20}{24.0}{\rmdefault}{\mddefault}{\itdefault}{$\overline{e}=3$}%
}}}}}
\put(17091,-726){\rotatebox{330.0}{\makebox(0,0)[lb]{\smash{{\SetFigFont{20}{24.0}{\rmdefault}{\mddefault}{\itdefault}{$\overline{e}=4$}%
}}}}}
\put(8506,-3961){\makebox(0,0)[lb]{\smash{{\SetFigFont{20}{24.0}{\rmdefault}{\mddefault}{\itdefault}{$c=0$}%
}}}}
\put(2761,-3676){\makebox(0,0)[lb]{\smash{{\SetFigFont{20}{24.0}{\rmdefault}{\mddefault}{\itdefault}{$c=0$}%
}}}}
\put(15451,-4696){\makebox(0,0)[lb]{\smash{{\SetFigFont{20}{24.0}{\rmdefault}{\mddefault}{\itdefault}{$c=0$}%
}}}}
\end{picture}%
}
\caption{The Aztec dragons of order (a) $1$, (b) $2$, (c) $3$, and (d) $5$.}
\label{Dragonnew}
\end{figure}

We consider a symmetric dissection of the plane into
equilateral triangles, squares, and regular hexagons, with 4 polygons meeting at
each vertex and with no two squares sharing an edge (see Figure \ref{Dragonnew}). We call the resulting lattice the \emph{dragon lattice}. On this lattice, James Propp introduced a family of regions called \textit{Aztec dragons} \cite{Propp}, as an analogue of \emph{Aztec diamonds} \cite{Elkies1,Elkies2}. Figures \ref{Dragonnew}(a)--(d) show the Aztec dragons of order $1$, $2$, $3$ and $5$, respectively (see the shaded regions). James Propp conjectured that the number of tilings of the Aztec dragon is always a power of $2$ (see Problem 15  in \cite{Propp}). The conjecture was proven by Ben Wieland (in unpublished work, as announced in \cite{Propp}) and Ciucu (see \cite[Corollary 7.2]{Ciucu}). In particular, they showed that:
\begin{thm}[Aztec Dragon Theorem]
The number of tilings of the Aztec dragon of order $n$ is $2^{n(n+1)}$.
\end{thm}

A \emph{perfect matching} of a graph $G$ is a collection of disjoint edges covering all vertices of $G$. The tilings of a region $R$ can be identified with the perfect matchings of its \emph{dual graph} (the graph whose vertices are fundamental regions in $R$ and whose edges connect precisely two fundamental regions sharing an edge).  In the view of this, we denote by $\M(G)$ the number of perfect matchings of a graph $G$. The dual graph of an Aztec dragon is called an \emph{Aztec dragon graph}.

\begin{figure}\centering
\includegraphics[width=12cm]{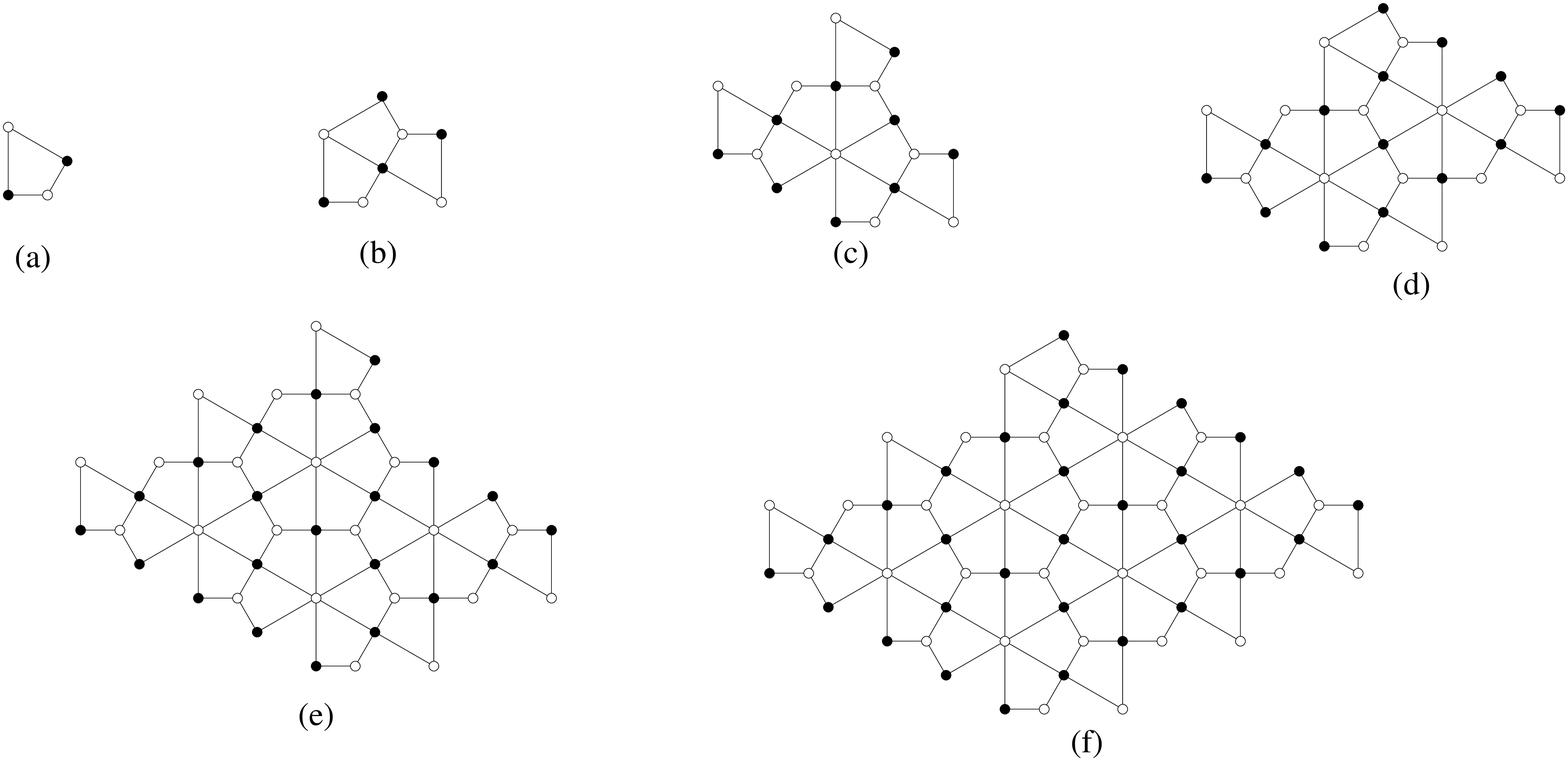}
\caption{The dragon graphs of order $n/2$ for $1\leq n\leq 6$.}
\label{Dragongraph}
\end{figure}

A further proof of the Aztec Dragon Theorem was provided by Cottrell and Young \cite{Cot}, by using the domino shuffling method on the Aztec dragon graphs. They also proved a related result for Aztec dragon graphs of half-integer order (see Figure \ref{Dragongraph}).  Recently, the Aztec dragon graphs have been investigated under the relations to cluster algebras (see \cite{Leo} and \cite{Zhang}). In \cite{Leo}, the authors introduced a larger family of $2$-parameter graphs called \emph{Aztec castles} and proved that the numbers of perfect matchings of these graphs are always some power of 2 (see Theorem 2.1 \cite{Leo}).

In this paper, we extend the family of the Aztec dragons to two new families of $6$-sided regions, which we call  \textit{dragon regions} (the precise definition of a dragon region will be provided in Section 2).
By using Kuo's graphical condensation method \cite{Kuo}, we prove that the number of tilings of a dragon region is always given by powers of $2$ and $3$ (see Theorem \ref{dragonthm}). The result generalizes the Aztec Dragon Theorem as well as the related results in \cite{Cot} and \cite{Leo}. Finally, in Section 5 we point out that our method can be used to obtain a related result when the tiles of the Aztec dragons carry some weights. Moreover, by the same technique, we obtain new families of regions having similar structure whose tilings are also enumerated by powers of $2$ and $3$.

\section{Dragon regions}

Our goal in this section is to generalize the family of Aztec dragons.

  The centers of the hexagonal fundamental regions on the dragon lattice form a triangular lattice. We consider a six-sided contour on this triangular lattice as follows.

  Starting from the center $S$ of a hexagon, we go $a$ units\footnote{The unit here is the smallest distance between the centers of two hexagonal fundamental regions in the dragon lattice, i.e the unit of the above triangular lattice.} southwest,  $b$ units southeast, $c$ units north, $\overline{d}$ units northeast, and $\overline{e}$ units northwest. After $5$ steps, we stop at the center $O$ of some hexagon. We adjust $\overline{e}$ so that $S$ and $O$ are on the same vertical line. Finally, we close the contour by go $\overline{f}$ units down or up depending on whether $O$ is above or below $S$ (see the dotted contour in Figures \ref{Dragonnew2}(a) and (b)).  The above choice of $\overline{e}$ requires
\begin{equation}\label{constrain1}
a+\overline{e}=b+\overline{d}.
\end{equation}

The vertex $O$ is above or below $S$ when $a$ is less or greater than $c+\overline{d}$, respectively. If $a>c+\overline{d}$, then the closure of the contour deduces that $\overline{f}=a-c-\overline{d}$; in the case where $a\leq c+\overline{d}$, the closure implies $\overline{f}=c+\overline{d}-a$. Thus, we always have
\begin{equation}\label{constrain2}
\overline{f}=|a-c-\overline{d}|.
\end{equation}

Next, we consider the region $R$ consisting of all fundamental regions that are restricted or intersected by the contour (see the union of all fundamental regions in Figures \ref{Dragonnew2}(a) and \ref{Dragonnew3}(a)). We now remove all the boundary squares of $R$. We also remove the boundary hexagons and the triangles adjacent to the boundary squares along the $b$- and $\overline{e}$-sides as well as the $\overline{f}$-side if  $a>c+\overline{d}$. The resulting region $R_1$ is illustrated as the shaded regions in Figures \ref{Dragonnew2}(a) and \ref{Dragonnew3}(a).

Color the plane black and white so that two fundamental regions sharing an edge have opposite color. Without loss of generality, we assume that all squares are colored black and all hexagon and triangles are colored white. It is easy to see that if a region admits a tiling, then the numbers of black and white fundamental regions of its must be equal. We say such a region is \emph{balanced}. The balance of the region $R_1$ yields
\begin{equation}\label{constrain3}
\overline{d}=2b-a-2c+1.
\end{equation}
Indeed, 
one can verify that the number of white fundamental regions in $R$ is one more than the number of black fundamental regions. Since $R_1$ is balanced, removing $R_1$ from $R$ does not alter the difference between the numbers of black and white fundamental regions. Enumerating exactly the black and white fundamental regions in $R-R_1$, we have $b+\overline{e}+\overline{f}+1=a+c+\overline{d}$ if $a>c+\overline{d}$ and $b+\overline{e}-\overline{f}+1=a+c+\overline{d}$ if $a\leq c+\overline{d}$.
Since $\overline{e}=b+\overline{d}-a$ and $\overline{f}=|a-c-\overline{d}|$, we obtain $\overline{d}=2b-a-2c+1$ in both cases.

At this point, we can write $\overline{d},\overline{e},\overline{f}$ all in terms of $a,b,c$ as
\begin{equation}
\overline{d}=2b-a-2c+1,
\end{equation}
 \begin{equation}
\overline{e}=b+\overline{d}-a=3b-2a-2c+1,
\end{equation}
\begin{equation}
\overline{f}=|a-c-\overline{d}|=|2b-2a-c+1|.
\end{equation}
This means that our contour and the corresponding region are determined by only three parameters $a,b,c$. We denote by $\mathcal{C}^{(1)}(a,b,c)$ the contour, and $DR^{(1)}_{a,b,c}$ the region $R_1$.

Next, we consider a variation $DR^{(2)}_{a,b,c}$ of the region $DR^{(1)}_{a,b,c}$ as follows. We define a six-sided contour similar to $\mathcal{C}^{(1)}_{a,b,c}$ by starting from the center $S$ of a hexagon, then going along  five sides of lengths $a,b,c,\widetilde{d},\widetilde{e}$, so that we end up at the center $O$ of a hexagon, that is vertically below or above $S$, and connecting $S$ and $O$ by a vertical side of length $\widetilde{f}$. Similar to the constraints (\ref{constrain1}) and (\ref{constrain2}) in the case of the contour $\mathcal{C}^{(1)}(a,b,c)$, we also have $a+\widetilde{e}=b+\widetilde{d}$ and $\widetilde{f}=|a-c-\widetilde{d}|$.

We consider the region $R'$ that contains all the fundamental regions that are restricted or intersected by the new contour (see Figures \ref{Dragonnew2}(b) and \ref{Dragonnew3}(b)).  Similar to the contour $\mathcal{C}^{(1)}(a,b,c)$, we remove all the boundary squares of $R'$. In addition, all the boundary hexagons and the triangles adjacent to the boundary squares along the $a$-, $c$-, and $\overline{d}$-sides, as well as the $\overline{f}$-side if $a\leq c+\overline{d}$, have been removed. The resulting region $R'_2$ is illustrated as the shaded regions in Figures \ref{Dragonnew2}(b) and \ref{Dragonnew3}(b).

Similar to the constraint (\ref{constrain3}), to guarantee that $R'_2$  is balanced, we must have $b+\widetilde{e}+\widetilde{f}-1=a+c+\widetilde{d}$ if $a>c+\widetilde{d}$, and $b+\widetilde{e}-\widetilde{f}-1=a+c+\widetilde{d}$ if $a\leq c+\widetilde{d}$.
Since $\widetilde{e}=b+\widetilde{d}-a$ and $\widetilde{f}=|a-c-\widetilde{d}|$, we have $\widetilde{d}=2b-a-2c-1$. Thus, $\widetilde{d},\widetilde{e},\widetilde{f}$ can be written in terms of $a,b,c$ as
\begin{equation}
\widetilde{d}=2b-a-2c-1,
\end{equation}
 \begin{equation}
\widetilde{e}=b+\widetilde{d}-a=3b-2a-2c-1,
\end{equation}
\begin{equation}
\widetilde{f}=|a-c-\widetilde{d}|=|2b-2a-c-1|.
\end{equation}
We denote by $\mathcal{C}^{(2)}(a,b,c)$ the contour, and $DR^{(2)}_{a,b,c}$ the region $R'_2$.

We call the two  regions  $DR^{(1)}_{a,b,c}$ and $DR^{(2)}_{a,b,c}$ \textit{dragon regions}, and the dual graphs of  the dragon regions  \textit{dragon graphs}. The numbers of tilings of the dragon regions are given by the theorem stated below.

\begin{figure}\centering
\setlength{\unitlength}{3947sp}%
\begingroup\makeatletter\ifx\SetFigFont\undefined%
\gdef\SetFigFont#1#2#3#4#5{%
  \reset@font\fontsize{#1}{#2pt}%
  \fontfamily{#3}\fontseries{#4}\fontshape{#5}%
  \selectfont}%
\fi\endgroup%
\resizebox{14cm}{!}{
\begin{picture}(0,0)%
\includegraphics{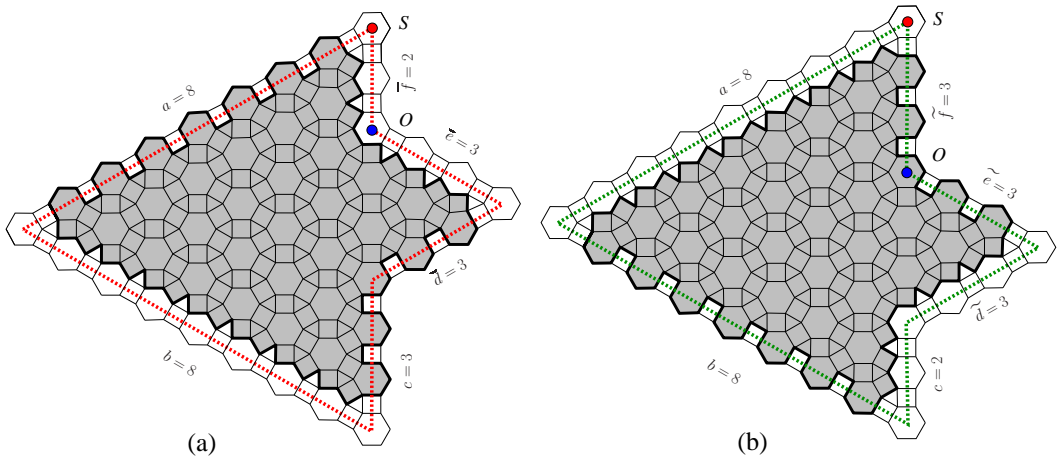}%
\end{picture}%

\begin{picture}(26712,11581)(2696,-11195)
\put(27520,-4116){\rotatebox{330.0}{\makebox(0,0)[lb]{\smash{{\SetFigFont{35}{40.8}{\rmdefault}{\mddefault}{\itdefault}{$\widetilde{e}=3$}%
}}}}}
\put(27417,-7693){\rotatebox{30.0}{\makebox(0,0)[lb]{\smash{{\SetFigFont{35}{40.8}{\rmdefault}{\mddefault}{\itdefault}{$\widetilde{d}=3$}%
}}}}}
\put(26502,-9458){\rotatebox{90.0}{\makebox(0,0)[lb]{\smash{{\SetFigFont{35}{40.8}{\rmdefault}{\mddefault}{\itdefault}{$c=2$}%
}}}}}
\put(26712,-2558){\rotatebox{90.0}{\makebox(0,0)[lb]{\smash{{\SetFigFont{35}{40.8}{\rmdefault}{\mddefault}{\itdefault}{$\widetilde{f}=3$}%
}}}}}
\put(20517,-8848){\rotatebox{330.0}{\makebox(0,0)[lb]{\smash{{\SetFigFont{35}{40.8}{\rmdefault}{\mddefault}{\itdefault}{$b=8$}%
}}}}}
\put(20892,-1933){\rotatebox{30.0}{\makebox(0,0)[lb]{\smash{{\SetFigFont{35}{40.8}{\rmdefault}{\mddefault}{\itdefault}{$a=8$}%
}}}}}
\put(13051,-1831){\rotatebox{90.0}{\makebox(0,0)[lb]{\smash{{\SetFigFont{35}{40.8}{\rmdefault}{\mddefault}{\itdefault}{$\overline{f}=2$}%
}}}}}
\put(13851,-3001){\rotatebox{330.0}{\makebox(0,0)[lb]{\smash{{\SetFigFont{35}{40.8}{\rmdefault}{\mddefault}{\itdefault}{$\overline{e}=3$}%
}}}}}
\put(6736,-2341){\rotatebox{30.0}{\makebox(0,0)[lb]{\smash{{\SetFigFont{35}{40.8}{\rmdefault}{\mddefault}{\itdefault}{$a=8$}%
}}}}}
\put(6681,-8631){\rotatebox{330.0}{\makebox(0,0)[lb]{\smash{{\SetFigFont{35}{40.8}{\rmdefault}{\mddefault}{\itdefault}{$b=8$}%
}}}}}
\put(13021,-9391){\rotatebox{90.0}{\makebox(0,0)[lb]{\smash{{\SetFigFont{35}{40.8}{\rmdefault}{\mddefault}{\itdefault}{$c=3$}%
}}}}}
\put(13651,-6771){\rotatebox{30.0}{\makebox(0,0)[lb]{\smash{{\SetFigFont{35}{40.8}{\rmdefault}{\mddefault}{\itdefault}{$\overline{d}=3$}%
}}}}}
\end{picture}}
\caption{The dragon regions (a) $DR^{(1)}_{8,8,3}$ and (b) $DR^{(2)}_{8,8,2}$.}
\label{Dragonnew2}
\end{figure}

\begin{figure}\centering
\setlength{\unitlength}{3947sp}%
\begingroup\makeatletter\ifx\SetFigFont\undefined%
\gdef\SetFigFont#1#2#3#4#5{%
  \reset@font\fontsize{#1}{#2pt}%
  \fontfamily{#3}\fontseries{#4}\fontshape{#5}%
  \selectfont}%
\fi\endgroup%
\resizebox{14cm}{!}{
\begin{picture}(0,0)%
\includegraphics{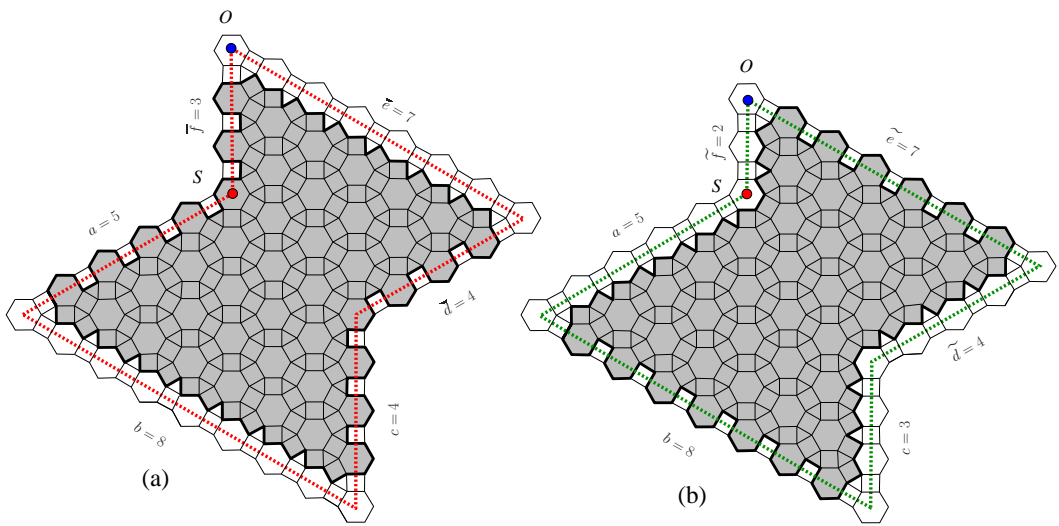}%
\end{picture}%
%
%

\begin{picture}(27907,13818)(2496,-13055)
\put(26506,-11236){\rotatebox{90.0}{\makebox(0,0)[lb]{\smash{{\SetFigFont{35}{40.8}{\rmdefault}{\mddefault}{\itdefault}{$c=3$}%
}}}}}
\put(18666,-5465){\rotatebox{30.0}{\makebox(0,0)[lb]{\smash{{\SetFigFont{35}{40.8}{\rmdefault}{\mddefault}{\itdefault}{$a=5$}%
}}}}}
\put(27706,-8761){\rotatebox{30.0}{\makebox(0,0)[lb]{\smash{{\SetFigFont{35}{40.8}{\rmdefault}{\mddefault}{\itdefault}{$\widetilde{d}=4$}%
}}}}}
\put(25861,-2896){\rotatebox{330.0}{\makebox(0,0)[lb]{\smash{{\SetFigFont{35}{40.8}{\rmdefault}{\mddefault}{\itdefault}{$\widetilde{e}=7$}%
}}}}}
\put(21541,-3421){\rotatebox{90.0}{\makebox(0,0)[lb]{\smash{{\SetFigFont{35}{40.8}{\rmdefault}{\mddefault}{\itdefault}{$\widetilde{f}=2$}%
}}}}}
\put(19927,-10876){\rotatebox{330.0}{\makebox(0,0)[lb]{\smash{{\SetFigFont{35}{40.8}{\rmdefault}{\mddefault}{\itdefault}{$b=8$}%
}}}}}
\put(7681,-2701){\rotatebox{90.0}{\makebox(0,0)[lb]{\smash{{\SetFigFont{35}{40.8}{\rmdefault}{\mddefault}{\itdefault}{$\overline{f}=3$}%
}}}}}
\put(12466,-1981){\rotatebox{330.0}{\makebox(0,0)[lb]{\smash{{\SetFigFont{35}{40.8}{\rmdefault}{\mddefault}{\itdefault}{$\overline{e}=7$}%
}}}}}
\put(14191,-7546){\rotatebox{30.0}{\makebox(0,0)[lb]{\smash{{\SetFigFont{35}{40.8}{\rmdefault}{\mddefault}{\itdefault}{$\overline{d}=4$}%
}}}}}
\put(4756,-5416){\rotatebox{30.0}{\makebox(0,0)[lb]{\smash{{\SetFigFont{35}{40.8}{\rmdefault}{\mddefault}{\itdefault}{$a=5$}%
}}}}}
\put(5881,-10591){\rotatebox{330.0}{\makebox(0,0)[lb]{\smash{{\SetFigFont{35}{40.8}{\rmdefault}{\mddefault}{\itdefault}{$b=8$}%
}}}}}
\put(12931,-10741){\rotatebox{90.0}{\makebox(0,0)[lb]{\smash{{\SetFigFont{35}{40.8}{\rmdefault}{\mddefault}{\itdefault}{$c=4$}%
}}}}}
\end{picture}}
\caption{The dragon regions (a) $DR^{(1)}_{5,8,4}$ and (b) $DR^{(2)}_{5,8,3}$.}
\label{Dragonnew3}
\end{figure}


\begin{thm}\label{dragonthm} Assume that $a$, $b$ and $c$ are three non-negative integers.

(a) If  $b\geq 2$, $2b-a-2c\geq-1$ and $3b-2a-2c\geq-1$, then
\begin{equation}\label{dragonequation1}
\M\left(DR^{(1)}_{a,b,c}\right)=2^{(b-c+1)(2b-a-c)+(a-b)^2}3^{\frac{(a-b+c)(a-b+c-1)}{2}}.
\end{equation}

(b) If  $b\geq 2$, $2b-a-2c\geq1$ and $3b-2a-2c\geq1$, then
\begin{equation}\label{dragonequation2}
\M\left(DR^{(2)}_{a,b,c}\right)=2^{(b-c-1)(2b-a-c)+(a-b)^2}3^{\frac{(a-b+c)(a-b+c+1)}{2}}.
\end{equation}
\end{thm}
We note that we assume the condition $b\geq 2$ in Theorem \ref{dragonthm} to guarantee that our regions are not empty.

\medskip

The Aztec dragon of order $n$ is exactly the region $DR^{(1)}_{n,n,0}$ (see Figure \ref{Dragonnew}), thus the Aztec dragon Theorem is a consequence of Theorem \ref{dragonthm}. Moreover, the Aztec dragon graph of order $n+\frac{1}{2}$ is isomorphic to the dual graph of the region $DR^{(2)}_{n+1,n+2,1}$, so  our theorem has Cottrel and Young's theorem (Theorem 1 in \cite{Cot}) as a special case.

We conclude this section by noticing that the Aztec castles in \cite{Leo} correspond to the dual graphs of the dragon regions $DR^{(i)}_{a,b,c}$'s, where $a-b+c$ is $-1$ or $0$.

\section{Recurrences for the numbers of tilings of dragon regions}

Consider the following system of recurrences. Here, we use the notations $\bigstar(a,b,c)$ and $\lozenge(a,b,c)$ for some functions from $\mathbb{Z}^3$ to $\mathbb{Z}$.

\medskip

\begin{equation}\tag{R1}
\begin{split}
\bigstar(a,b,c)\bigstar(a-3,b-3,c-2)=\bigstar(a-2,b-1,c)\bigstar(a-1,b-2,c-2)\\+\bigstar(a-1,b-1,c-1)\bigstar(a-2,b-2,c-1).
\end{split} \label{R1}
\end{equation}

\medskip

\begin{equation}\tag{R2}
\bigstar(a,b,c)\bigstar(a-2,b-2,c)=\bigstar(a-1,b-1,c)^2+\bigstar(a,b,c+1)\bigstar(a-2,b-2,c-1).
\label{R2}
\end{equation}

\medskip

\begin{equation}\tag{R3}
\begin{cases}
 \begin{split}
 \bigstar(a,b,0)\bigstar(a-2,b-2,0)=\bigstar(a-1,b-1,0)^2\\+\bigstar(a,b,1)\lozenge(3b-2a+1,2b-a+1,1);\end{split}\\
 \begin{split}\lozenge(a,b,0)\lozenge(a-2,b-2,0)=\lozenge(a-1,b-1,0)^2\\
 +\lozenge(a,b,1)\bigstar(3b-2a-1,2b-a-1,1). \end{split}
 \end{cases}
\label{R3}
\end{equation}

\medskip

\begin{equation}\tag{R4}
  \begin{split}
        \bigstar(a,b,c)\bigstar(a-2,b-3,c-2)=\bigstar(a-1,b-1,c)\bigstar(a-1,b-2,c-2)\\+\bigstar(a-2,b-2,c-1)\bigstar(a,b-1,c-1).
 \end{split}\label{R4}
\end{equation}

\medskip

\begin{equation}\tag{R5}
    \begin{split}
      \bigstar(a,b,c)\bigstar(a-2,b-3,c-2)=\bigstar(c,b-1,a-1)\bigstar(a-1,b-2,c-2)\\+\bigstar(a-2,b-2,c-1)\bigstar(a,b-1,c-1).
    \end{split}
    \label{R5}
\end{equation}

In the next part of the section, we will show that $\M\left(DR^{(1)}_{a,b,c}\right)$ and $\M\left(DR^{(2)}_{a,b,c}\right)$ satisfy the above recurrences (with certain constraints).  Our proofs are based on the following Kuo's Condensation Theorem.

\begin{thm} [Kuo \cite{Kuo}] \label{kuothm} Let $G=(V_1,V_2,E)$ be a planar bipartite graph in which $|V_1|=|V_2|$. Assume that $u,v,w$ and $t$ are four vertices appearing in a cyclic order on a face of $G$ with $u,w\in V_1$ and $v,t\in V_2$. Then
\begin{align}\label{kuoeq}
\M(G)\M(G-\{u,v,w,t\})=&\M(G-\{u,v\})\M(G-\{w,t\})\notag\\
&+\M(G-\{t,u\})\M(G-\{v,w\}).
\end{align}
\end{thm}

\begin{lem}\label{dragonlem1}
Let $a,b,c$ be non-negative integers so that $b\geq5$ and $c\geq 2$.  Assume in addition that $\overline{d}:=2b-a-2c+1\geq0$, $\overline{e}:=3b-2a-2c+1\geq 0$, and $a\geq c+\overline{d}+1$. Then $\M\left(DR^{(1)}_{a,b,c}\right)$ satisfies the recurrence (\ref{R1}).

Analogously,  $\M\left(DR^{(2)}_{a,b,c}\right)$ satisfies the recurrence (\ref{R1}), when $\widetilde{d}:=2b-a-2c-1\geq0$, $\widetilde{e}:=3b-2a-2c-1\geq 0$, and $a\geq c+\widetilde{d}+1$.
\end{lem}

\begin{figure}\centering
\includegraphics[width=12cm]{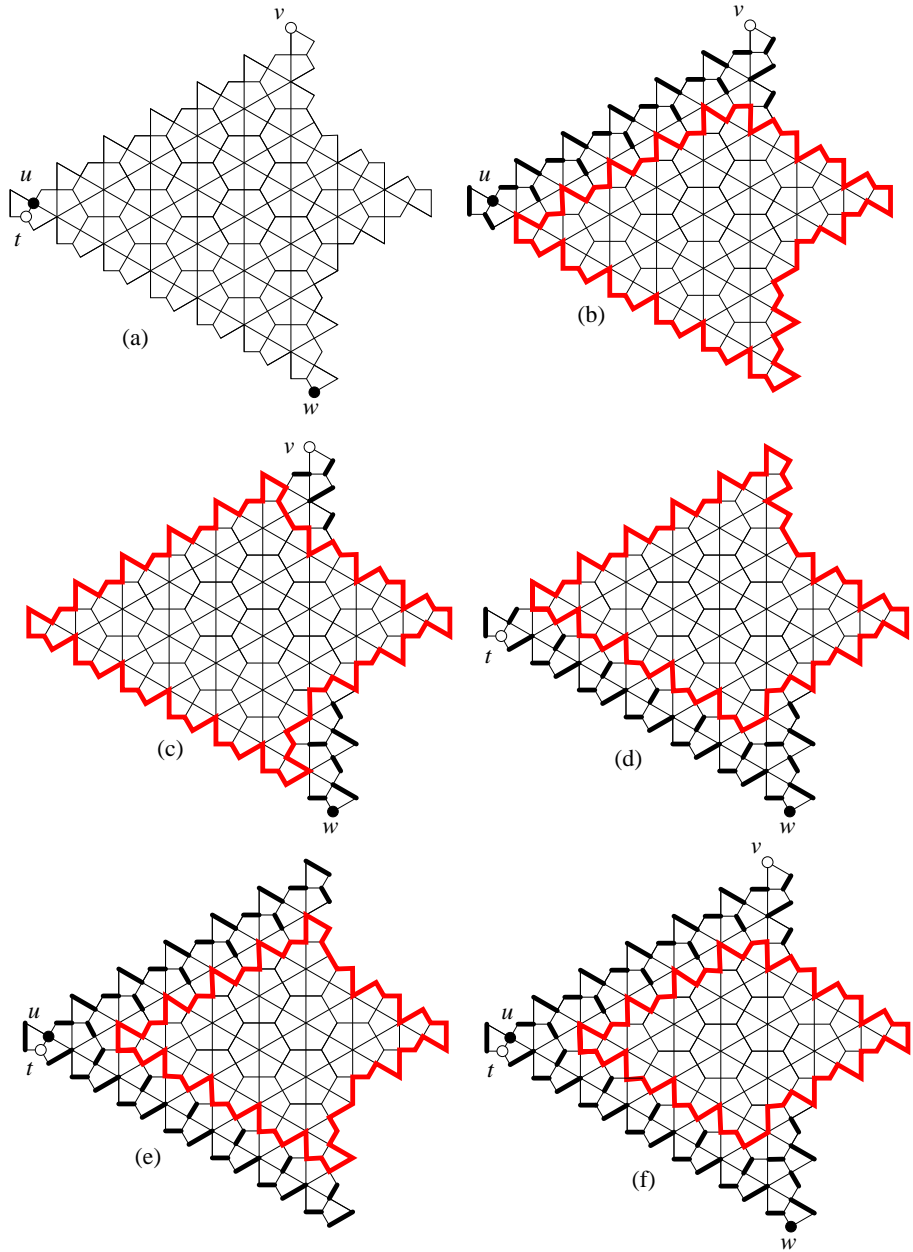}
\caption{Verifying $\M(DR^{(1)}_{a,b,c})$ satisfies the recurrence (\ref{R1}). The dark bold edges indicate the forced ones.}
\label{Dr6}
\end{figure}

\begin{proof}
To prove that $\M(DR^{(1)}_{a,b,c})$ satisfies the recurrence (\ref{R1}), we need to show that
\begin{equation}
\begin{split}
\M(DR^{(1)}_{a,b,c})\M(DR^{(1)}_{a-3,b-3,c-2})=\M(DR^{(1)}_{a-2,b-1,c})\M(DR^{(1)}_{a-1,b-2,c-2})\\+\M(DR^{(1)}_{a-1,b-1,c-1})\M(DR^{(1)}_{a-2,b-2,c-1}).
\end{split} \label{dragoneq1}
\end{equation}
We now apply the Kuo's  Condensation Theorem \ref{kuothm} to the dual graph $G$ of  the regions $DR^{(1)}_{a,b,c}$ with the four vertices $u,v,w,t$ chosen as in Figure \ref{Dr6}(a). More precisely, we chose $u$ and $t$ on the west corner, $v$ on the north corner, and $w$ on the south corner of the graph.  Each of the graphs $G-\{u,v,w,t\}$, $G-\{u,v\}$, $G-\{w,t\}$, $G-\{t,u\}$, and $G-\{v,w\}$ has some edges that are forced in any perfect matchings. Fortunately, by removing these edges, we get a new dragon graph having the same number of perfect matchings as the original graph. In particular, after removing forced edges from $G-\{u,v\}$, we get the dual graph of $DR^{(1)}_{a-2,b-1,c}$ (see Figure \ref{Dr6}(b); the forced edges are the dark bold ones; and the dual graph of $DR^{(1)}_{a-2,b-1,c}$ is illustrated by the one restricted by the light bold contour). Thus, we have
\begin{equation}\label{dragoneq2}
\M(G-\{u,v\})=\M(DR^{(1)}_{a-2,b-1,c}).
\end{equation}
Similarly, we get the following equalities:
\begin{equation}\label{dragoneq3}
\M(G-\{v,w\})=\M(DR^{(1)}_{a-1,b-1,c-1}) \text{ (see Figure \ref{Dr6}(c))},
\end{equation}
\begin{equation}\label{dragoneq4}
\M(G-\{w,t\})=\M(DR^{(1)}_{a-1,b-2,c-2})  \text{ (see Figure \ref{Dr6}(d))},
\end{equation}
\begin{equation}\label{dragoneq5}
\M(G-\{t,u\})=\M(DR^{(1)}_{a-2,b-2,c-1})  \text{ (see Figure \ref{Dr6}(e))},
\end{equation}
and
\begin{equation}\label{dragoneq6}
\M(G-\{u,v,w,t\})=\M(DR^{(1)}_{a-3,b-3,c-2})  \text{ (see Figure \ref{Dr6}(f))}.
\end{equation}

Substituting the above five equalities (\ref{dragoneq2})--(\ref{dragoneq6}) into the equation (\ref{kuoeq}) in Kuo's Theorem \ref{kuothm}, we obtain (\ref{dragoneq1}).

The statement for the region $DR^{(2)}_{a,b,c}$ can be obtained similarly.
\end{proof}

\begin{figure}\centering
\includegraphics[width=12cm]{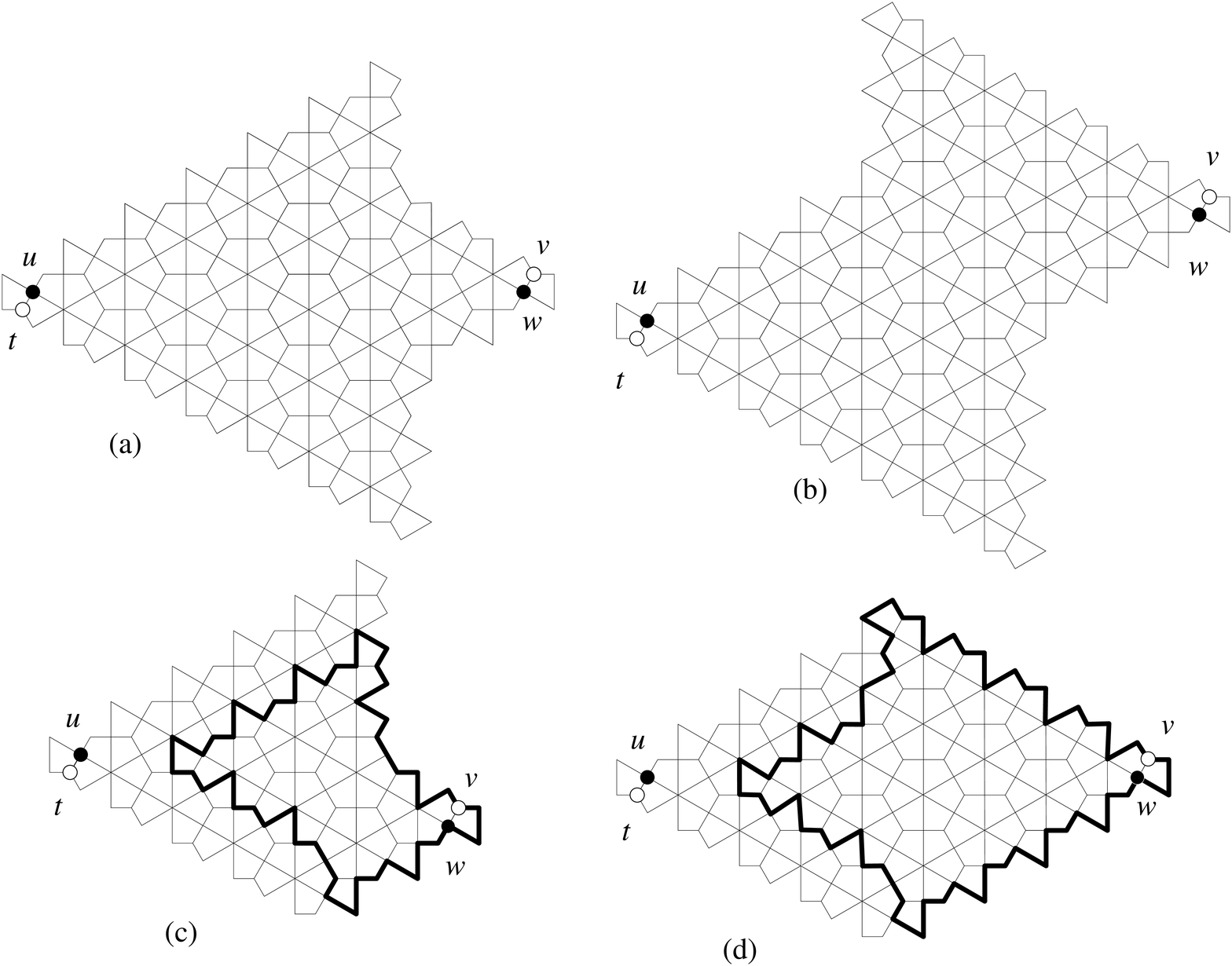}
\caption{The dual graphs of (a) $DR^{(1)}_{8,8,3}$, (b) $DR^{(1)}_{5,8,4}$, (c) $DR^{(1)}_{7,5,0}$ and (d) $DR^{(1)}_{6,5,0}$, together with the selection of the four vertices $u,v,w,t$.}
\label{DR5}
\end{figure}

\begin{lem}\label{dragonlem2}
Let $a,b,c$ be non-negative integers so that $a\geq2$ and $b\geq 4$.

(a) Assume that $c\geq1$.  Then $\M(DR^{(1)}_{a,b,c})$ satisfies the recurrence (\ref{R2}), when $\overline{d}:=2b-a-2c+1\geq 2$ and $\overline{e}:=3b-2a-2c+1\geq 2$. Analogously, $\M(DR^{(2)}_{a,b,c})$ satisfies the recurrence (\ref{R2}), when $\widetilde{d}:=2b-a-2c-1\geq 2$ and $\widetilde{e}:=3b-2a-2c-1\geq 2$.

(b) Assume that $c=0$.  Then the pair $\left(\M(DR^{(1)}_{a,b,c}),\right.$ $\left.\M(DR^{(2)}_{a,b,c})\right)$ satisfies the first equality in the recurrence  (\ref{R3}) for $(\bigstar,\lozenge)$, when $\overline{d}:=2b-a-2c+1\geq 2$ and $\overline{e}:=3b-2a-2c+1\geq 2$, i.e.
 \begin{equation}\label{dragobeq3-4}
     \begin{split}
      \M\left(DR^{(1)}_{a,b,0}\right)\M\left(DR^{(1)}_{a-2,b-2,0}\right)=
      \M\left(DR^{(1)}_{a-1,b-1,0}\right)^2\\+\M\left(DR^{(1)}_{a,b,1}\right)\M\left(DR^{(2)}_{3b-2a+1,2b-a+1,1}\right),
      \end{split}
\end{equation}
 and  it satisfies the second equality of (\ref{R3}), when $\widetilde{d}:=2b-a-2c-1\geq 2$ and $\widetilde{e}:=3b-2a-2c-1\geq 2$, i.e.
    \begin{equation}\label{dragoneq4-4}
     \begin{split}
      \M\left(DR^{(2)}_{a,b,0}\right)\M\left(DR^{(2)}_{a-2,b-2,0}\right)=\M\left(DR^{(2)}_{a-1,b-1,0}\right)^2\\+\M\left(DR^{(2)}_{a,b,1}\right)\M\left(DR^{(1)}_{3b-2a-1,2b-a-1,1}\right).
      \end{split}
   \end{equation}
\end{lem}

\begin{proof}
(a) Similar to Lemma \ref{dragonlem1}, we apply Kuo's Theorem \ref{kuothm} to the dual graph $G$ of $DR^{(1)}_{a,b,c}$ with the four vertices $u,v,w,t$ chosen as in Figures \ref{DR5}(a) and (b), for the cases $a>c+d$ and $a\leq c+d$, respectively. By considering forced edges, we get in both cases that
\begin{equation}\label{dragoneq7}
\M(G-\{u,v\})=\M\left(DR^{(1)}_{a-1,b-1,c}\right),
\end{equation}
\begin{equation}\label{dragoneq8}
\M(G-\{v,w\})=\M\left(DR^{(1)}_{a,b,c+1}\right),
\end{equation}
\begin{equation}\label{dragoneq9}
\M(G-\{w,t\})=\M\left(DR^{(1)}_{a-1,b-1,c}\right),
\end{equation}
\begin{equation}\label{dragoneq10}
\M(G-\{t,u\})=\M\left(DR^{(1)}_{a-2,b-2,c-1}\right),
\end{equation}
and
\begin{equation}\label{dragoneq11}
\M(G-\{u,v,w,t\})=\M\left(DR^{(1)}_{a-2,b-2,c}\right).
\end{equation}
Again, substituting (\ref{dragoneq7})--(\ref{dragoneq11}) into the equation (\ref{kuoeq}) in Kuo's Theorem \ref{kuothm}, we deduce
\begin{align}
\M\left(DR^{(1)}_{a,b,c}\right)\M\left(DR^{(1)}_{a-2,b-2,c}\right)=&\M\left(DR^{(1)}_{a-1,b-1,c}\right)^2\notag\\
&+\M\left(DR^{(1)}_{a,b,c+1}\right)\M\left(DR^{(1)}_{a-2,b-2,c-1}\right),
\end{align}
which implies that $\M\left(DR^{(1)}_{a,b,c}\right)$ satisfies the recurrence (\ref{R2}). Similarly, we can prove that the number tilings of $DR^{(2)}_{a,b,c}$ satisfies  (\ref{R2}).

\medskip

(b) We present here the proof of (\ref{dragobeq3-4}),  and (\ref{dragoneq4-4}) can be treated in the same way.

Similar to part (a), we apply the Kuo's Theorem to the dual graph $G$ of $DR^{(1)}_{a,b,0}$, as in Figures \ref{DR5}(c) and (d) for the cases $a>c+\overline{d}$ and $a\leq c+\overline{d}$, respectively. We still get the equalities (\ref{dragoneq7}), (\ref{dragoneq8}), (\ref{dragoneq9}), and (\ref{dragoneq11}) as in part (a). However, after removing forced edges from the graph $G-\{t,u\}$, we do \textit{not} get the dual graph of the region $DR^{(1)}_{a-2,b-2,c-1}$ any more (this region does \textit{not} even exist when $c=0$). By reflecting the resulting graph about a vertical line, we get the dual graph of $DR^{(2)}_{\overline{e},\overline{d},1}$ (see the graphs restricted by the bold contours in Figures \ref{DR5}(c) and (d)). Thus, (\ref{dragobeq3-4}) follows from Kuo's Theorem \ref{kuothm}. 
\end{proof}

\begin{figure}\centering
\includegraphics[width=12cm]{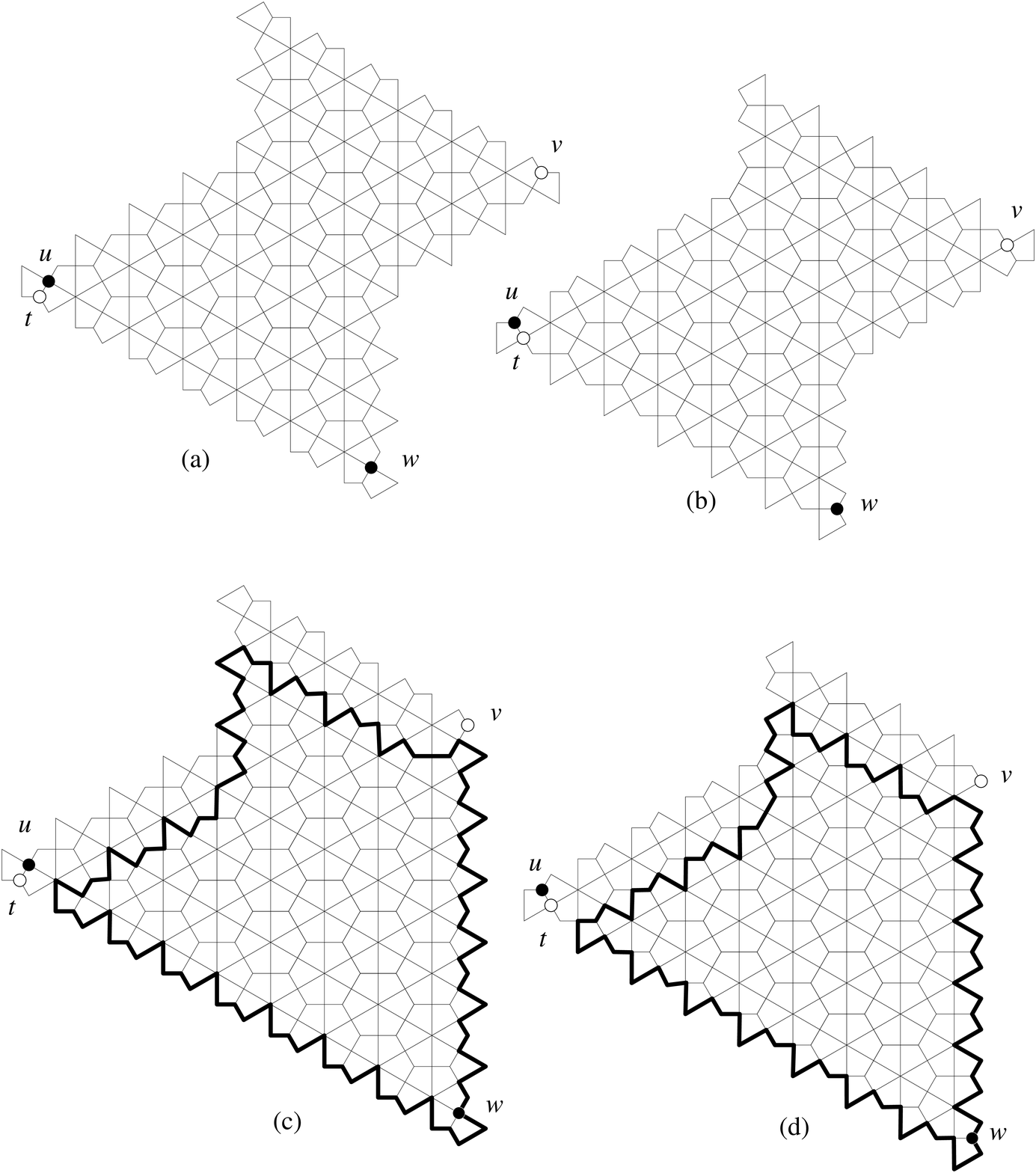}
\caption{The dual graphs of (a) $DR^{(1)}_{5,8,4}$, (b) $DR^{(2)}_{5,8,3}$, (c) $DR^{(1)}_{5,10,8}$ and (d) $DR^{(2)}_{5,10,7}$, together with the selection of the four vertices $u,v,w,t$.}
\label{DR3}
\end{figure}

\begin{lem}\label{dragonlem3}
Assume that $a,b,c$ are non-negative integers so that $a\geq 2$, $b\geq 5$ and $c\geq2$. Denote by $\overline{d}:=2b-a-2c+1$, $\overline{e}:=3b-2a-2c+1$, $\widetilde{d}:=2b-a-2c-1$, and $\widetilde{e}:=3b-2a-2c-1$ as usual.

(a) Assume in addition that $\overline{d}\geq1$, $\overline{e}\geq0$, and $a\leq c+\overline{d}$, then $\M\left(DR^{(1)}_{a,b,c}\right)$ satisfies the recurrence $(\ref{R4})$.  The number of tilings $\M\left(DR^{(2)}_{a,b,c}\right)$ also satisfies the recurrence $(\ref{R4})$, when $\widetilde{d}\geq1$, $\widetilde{e}\geq0$, and $a\leq c+\widetilde{d}$.

(b) If $\overline{d}=0$, $\overline{e}\geq0$, and $a\leq c+\overline{d}$, then  $\M\left(DR^{(1)}_{a,b,c}\right)$ satisfies the recurrence $(\ref{R5})$. Analogously,
  $\M\left(DR^{(2)}_{a,b,c}\right)$ also satisfies the recurrence $(\ref{R5})$, when $\widetilde{d}=0$, $\widetilde{e}\geq0$, and $a\leq c+\widetilde{d}$.
\end{lem}

\begin{proof}
(a) To prove that  $\M\left(DR^{(i)}_{a,b,c}\right)$ satisfies (\ref{R4}) (with the corresponding constraints), for $i=1,2$,  we apply Kuo's  Theorem \ref{kuothm} to the dual graph $G$ of the region $DR^{(i)}_{a,b,c}$ with the four vertices $u,v,w,t$ chosen as in Figures \ref{DR3}(a) and  (b) (for the case $i=1$ and $2$, respectively). We need to show that
\begin{equation}
  \begin{split}
   \M\left(DR^{(i)}_{a,b,c}\right)\M\left(DR^{(i)}_{a-2,b-3,c-2}\right)=\M\left(DR^{(i)}_{a-1,b-1,c}\right)\M\left(DR^{(i)}_{a-1,b-2,c-2}\right)\\+\M\left(DR^{(i)}_{a-2,b-2,c-1}\right)
  \M\left(DR^{(i)}_{a,b-1,c-1}\right).
 \end{split}\label{dragoneq12}
\end{equation}

Similar to  the proofs of Lemmas \ref{dragonlem1} and \ref{dragonlem2}, by  removing edges forced by the removal of the vertices $u,v,w,t$, we get new dragon graphs. To precise, we have
\begin{equation}\label{dragoneq13}
\M(G-\{u,v\})=\M\left(DR^{(i)}_{a-1,b-1,c}\right),
\end{equation}
\begin{equation}\label{dragoneq14}
\M(G-\{v,w\})=\M\left(DR^{(i)}_{a,b-1,c-1}\right),
\end{equation}
\begin{equation}\label{dragoneq15}
\M(G-\{w,t\})=\M\left(DR^{(i)}_{a-1,b-2,c-2}\right),
\end{equation}
\begin{equation}\label{dragoneq16}
\M(G-\{t,u\})=\M\left(DR^{(i)}_{a-2,b-2,c-1}\right),
\end{equation}
and
\begin{equation}\label{dragoneq17}
\M(G-\{u,v,w,t\})=\M\left(DR^{(i)}_{a-2,b-3,c-2}\right).
\end{equation}
Then the equality (\ref{dragoneq12}) follows also from Kuo's Theorem \ref{kuothm}.

(b) First, we note that if  $\overline{d}:=2b-a-2c+1=0$, then $2(b-1)-(a-1)-2c+1=-1$. This means that the region $DR^{(1)}_{a-1,b-1,c}$ does \emph{not} exist when $\overline{d}=0$. Similarly, the region $DR^{(2)}_{a-1,b-1,c}$ does not exist when $\widetilde{d}=0$. In particular, the equality (\ref{dragoneq13}) in part (a) is \textit{not} true anymore when $\overline{d}=0$ or $\widetilde{d}=0$.

 This part can be treated similarly to part (a). We still apply the Kuo's Theorem to the dual graph $G$ of $DR^{(i)}_{a,b,c}$ with the four vertices $u,v,w,t$ selected as in Figures \ref{DR3}(c) and (d). We also obtain the five equalities (\ref{dragoneq14})--(\ref{dragoneq17}) as in part (a). However, after removing forced edges from the graph $G-\{u,v\}$, we reflect the resulting graph about a vertical line and get the dual graph of $DR^{(i)}_{c,b-1,a-1}$ (see the graphs restricted by the bold contours in Figures \ref{DR3}(c) and (d)). Thus,  we have
 \begin{equation}\label{dragoneq18}
\M(G-\{u,v\})=\M\left(DR^{(i)}_{c,b-1,a-1}\right)
\end{equation}
 instead of (\ref{dragoneq13}) in part (a). By substituting (\ref{dragoneq14})--(\ref{dragoneq18}) into (\ref{kuoeq}), we get the statement in part (b). 
\end{proof}

Denote respectively by $\Phi(a,b,c)$ and $\Psi(a,b,c)$ the expressions on the right-hand sides of the equalities (\ref{dragoneq1}) and (\ref{dragoneq2}), i.e.
\begin{equation}
\Phi(a,b,c):=2^{(b-c+1)(2b-a-c)+(a-b)^2}3^{\frac{(a-b+c)(a-b+c-1)}{2}}
\end{equation}
and
\begin{equation}
\Psi(a,b,c):=2^{(b-c-1)(2b-a-c)+(a-b)^2}3^{\frac{(a-b+c)(a-b+c+1)}{2}}.
\end{equation}
It is straightforward to verify the following result.
\begin{lem}\label{lemnew}
For any integers $a,b,c$ the functions $\Phi(a,b,c)$ and $\Psi(a,b,c)$ satisfy all the recurrences (\ref{R1})--(\ref{R5}). In particular,   we have
\begin{equation}
\begin{split}
\Phi(a,b,c)\Phi(a-3,b-3,c-2)=\Phi(a-2,b-1,c)\Phi(a-1,b-2,c-2)\\+\Phi(a-1,b-1,c-1)\Phi(a-2,b-2,c-1),
\end{split} \label{R1n}
\end{equation}
\begin{equation}
\begin{split}
\Psi(a,b,c)\Phi(a-3,b-3,c-2)=\Psi(a-2,b-1,c)\Psi(a-1,b-2,c-2)\\+\Psi(a-1,b-1,c-1)\Psi(a-2,b-2,c-1),
\end{split} \label{R1m}
\end{equation}
\begin{equation}
\Phi(a,b,c)\Phi(a-2,b-2,c)=\Phi(a-1,b-1,c)^2+\Phi(a,b,c+1)\Phi(a-2,b-2,c-1),
\label{R2n}
\end{equation}
\begin{equation}
\Psi(a,b,c)\Psi(a-2,b-2,c)=\Psi(a-1,b-1,c)^2+\Psi(a,b,c+1)\Psi(a-2,b-2,c-1),
\label{R2m}
\end{equation}
\begin{equation}
\begin{cases}
\begin{split}
 \Phi(a,b,0)\Phi(a-2,b-2,0)=\Phi(a-1,b-1,0)^2\\+\Phi(a,b,1)\Psi(3b-2a+1,2b-a+1,1);\end{split}\\
 \begin{split}
\Psi(a,b,0)\Psi(a-2,b-2,0)=\Psi(a-1,b-1,0)^2\\+\Psi(a,b,1)\Phi(3b-2a-1,2b-a-1,1),
\end{split}
 \end{cases}
\label{R3n}
\end{equation}
\begin{equation}
  \begin{split}
        \Phi(a,b,c)\Phi(a-2,b-3,c-2)=\Phi(a-1,b-1,c)\Phi(a-1,b-2,c-2)\\+\Phi(a-2,b-2,c-1)\Phi(a,b-1,c-1),
 \end{split}\label{R4n}
\end{equation}
\begin{equation}
  \begin{split}
        \Psi(a,b,c)\Psi(a-2,b-3,c-2)=\Psi(a-1,b-1,c)\Psi(a-1,b-2,c-2)\\+\Psi(a-2,b-2,c-1)\Psi(a,b-1,c-1),
 \end{split}\label{R4m}
\end{equation}
\begin{equation}
    \begin{split}
      \Phi(a,b,c)\Phi(a-2,b-3,c-2)=\Phi(c,b-1,a-1)\Phi(a-1,b-2,c-2)\\+\Phi(a-2,b-2,c-1)\Phi(a,b-1,c-1),
    \end{split}
    \label{R5n}
\end{equation}
\begin{equation}
    \begin{split}
      \Psi(a,b,c)\Psi(a-2,b-3,c-2)=\Psi(c,b-1,a-1)\Psi(a-1,b-2,c-2)\\+\Psi(a-2,b-2,c-1)\Psi(a,b-1,c-1).
    \end{split}
    \label{R5m}
\end{equation}
\end{lem}

\section{Proof of Theorem \ref{dragonthm}}

We show that
\begin{equation}\label{dram1}
\M\left(DR^{(1)}_{a,b,c}\right)=\Phi(a,b,c),
\end{equation}
for $b\geq 2,$ $\overline{d}\geq0,\overline{e}\geq 0$,
and that
\begin{equation}\label{dram2}
\M\left(DR^{(2)}_{a,b,c}\right)=\Psi(a,b,c),
\end{equation}
for $b\geq 2$, $\widetilde{d}\geq0,$ $\widetilde{e}\geq 0$,
 by induction on the perimeters of the contours of the regions. In this proof, we always denote by $\overline{d}:=2b-a-2c+1$, $\overline{e}:=3b-2a-2c+1$, $\overline{f}:=|2b-2a-c+1|$, $\widetilde{d}:=2b-a-2c-1$, $\widetilde{e}:=3b-2a-2c-1$,  and $\widetilde{f}:=|2b-2a-c-1|$ as usual.

Denote by $\mathcal{P}^{(i)}(a,b,c)$ the perimeter of the contour $\mathcal{C}^{(i)}(a,b,c)$, for $i=1,2$. One readily sees that $\mathcal{P}^{(1)}(a,b,c)=a+b+c+\overline{d}+\overline{e}+\overline{f}=4b-2c+1$ if $a>c+\overline{d}$, and $8b-4a-4c+3$ otherwise. Similarly,  $\mathcal{P}^{(2)}(a,b,c)$ equals $4b-2c-1$ if $a>c+\widetilde{d}$, and $8b-4a-4c-3$ otherwise. In particular, the two parameters are both odd.

\medskip

The base cases are all $DR^{(1)}$-regions with the triple $(a,b,c)$ satisfying \emph{at least one of} the following conditions:
\begin{enumerate}
\item[\quad \quad(i)] $\mathcal{P}^{(1)}(a,b,c)\leq 15$
\item[(ii)] $b \leq 4$
\item[(iii)] $c+\overline{d} \leq 2$;
\end{enumerate}
and all $DR^{(2)}$-regions having the triple $(a,b,c)$ satisfying at least one of the following conditions:
\begin{enumerate}
\item[(iv)] $\mathcal{P}^{(2)}(a,b,c)\leq 15$
\item[(v)] $b \leq 3$
\item[(vi)] $c+\widetilde{d} \leq 2$,
\end{enumerate}

It is easy to see that each of the conditions (i), (ii) and (iii) implies some small upper bounds for all $a,b,c$. Indeed, if (i) holds, we get from the triangle inequality that $2a,2b,2c\leq \mathcal{P}^{(1)}(a,b,c)\leq 15$. Thus, we get $a,b,c\leq 7$. Similarly, if $b\leq 4$, then $\overline{d}=2b-a-2c+1\geq0$, this implies $c \leq 4$. By the same reason, $ a\leq 2b+1 \leq 9$. Finally, if $c+\overline{d}\leq 2$, then $2b\leq a+2c+1\leq a+5\leq (b+\overline{d})+5\leq b+7$. Thus $b\leq 7$, so $a\leq b+ \overline{d}\leq 9$.

 Taking into account also the inequalities $b\geq 2$, $\overline{d}=2b-a-2c+1\geq 0$ and $\overline{e}=3b-2a-2c+1\geq 0$, we get that there are 53 $DR^{(1)}$-regions in the base cases. Similarly, there are $28$ $DR^{(2)}$-regions with the triple $(a,b,c)$ satisfying at least one of the conditions (iv), (v), and (vi). For each of the dragon regions in the base cases, (\ref{dram1}) and (\ref{dram2}) can be readily checked, for example, by a computer package \texttt{vaxmacs}\footnote{This software is available at \texttt{http://dwilson.com/vaxmacs}.} written by David Wilson. The number of tilings of each of these dragon regions is returned in a second.

For the induction step, we assume that $k$ is an odd integer greater than 15 and that (\ref{dram1}) and (\ref{dram2}) hold for any $DR^{(i)}$-regions with perimeter less than $k$ ($k\geq 17$).
By the base cases,  we only need to show that (\ref{dram1}) holds for any triples $(a,b,c)$ in the domain
\[\overline{\mathcal{D}}=\{(a,b,c)\in \mathbb{Z}^3|\,\mathcal{P}^{(1)}(a,b,c)=k, b\geq 5,\overline{d}\geq0,\overline{e}\geq 0,c+\overline{d}\geq3\},\]
and that (\ref{dram2}) holds for any triples $(a,b,c)$ in the domain
\[\widetilde{\mathcal{D}}=\{(a,b,c)\in \mathbb{Z}^3|\, \mathcal{P}^{(2)}(a,b,c)=k, b\geq 4,\widetilde{d}\geq0,\widetilde{e}\geq 0,c+\widetilde{d}\geq3\}.\]

Next, we partition each of $\overline{\mathcal{D}}$ and $\widetilde{\mathcal{D}}$ into four subdomains as follows.
\[\overline{\mathcal{D}}_1=\overline{\mathcal{D}}\cap\{(a,b,c)|\,2\leq a\leq c+\overline{d}\};\quad \widetilde{\mathcal{D}}_1=\widetilde{\mathcal{D}}\cap\{(a,b,c)|\,2\leq a\leq c+\widetilde{d}\};\]
\[\overline{\mathcal{D}}_2=\overline{\mathcal{D}}\cap\{(a,b,c)|\,a\leq1\};\quad \widetilde{\mathcal{D}}_2=\widetilde{\mathcal{D}}\cap\{(a,b,c)|\,a\leq 1\};\]
\[\overline{\mathcal{D}}_3=\overline{\mathcal{D}}\cap\{(a,b,c)|\,a>c+\overline{d}, \overline{e}\geq \overline{d}\};\quad \widetilde{\mathcal{D}}_3=\widetilde{\mathcal{D}}\cap\{(a,b,c)|\,a>c+\widetilde{d}, \widetilde{e}\geq \widetilde{d}\};\]
\[\overline{\mathcal{D}}_4=\overline{\mathcal{D}}\cap\{(a,b,c)|\,a>c+\overline{d}, \overline{e}< \overline{d}\};\quad \widetilde{\mathcal{D}}_4=\widetilde{\mathcal{D}}\cap\{(a,b,c)|\,a>c+\widetilde{d}, \widetilde{e}< \widetilde{d}\}.\]
We will show that $\M\left(DR^{(1)}_{a,b,c}\right)$ and $\Phi(a,b,c)$, and $\M\left(DR^{(2)}_{a,b,c}\right)$ and $\Psi(a,b,c)$ agree in each of the above subdomains.

There are four cases to distinguish.

\medskip

\quad \textit{Case 1: $(a,b,c)\in \overline{\mathcal{D}}_1\cup \widetilde{\mathcal{D}}_1$.}

\medskip
First, we assume that $(a,b,c)\in \overline{\mathcal{D}}_1$.

We notice that the condition $b\geq 5$ guarantees that   $\overline{e} \geq 2$.  Indeed, since we are assuming that $a\leq c+\overline{d}$, we have $a\leq 2b-a-c+1$. By $\overline{d}\geq 0$, we get $2b-a-2c+1\geq 0$. Adding the above two inequalities, we get $3a+3c\leq4b+2$. Then, $\overline{e}=3b-2a-2c+1=3b-\frac{2}{3}(3a+3c)+1\geq \frac{b}{3}-\frac{1}{3}\geq \frac{4}{3}$. Since $\overline{e}$ is an integer, $\overline{e}\geq 2$.

Moreover, since we are assuming that $c+\overline{d}\geq 3$, then at least one of $c$ and $\overline{d}$ is greater than 2. We first assume that $\overline{d}\geq 2$.

\medskip

\emph{Case 1.a. $\overline{d}\geq 2$.}

\medskip

If $c\geq 1$, then, by Lemmas \ref{dragonlem2}(a) and \ref{lemnew}, $\M(DR^{(1)}_{a,b,c})$ and $\Phi(a,b,c)$ both satisfy the recurrence (\ref{R2}). One notices that all regions other than $DR^{(1)}_{a,b,c}$ in the equality (\ref{dragoneq1}) have perimeters strictly less than $k$. Thus, by induction hypothesis, we get
\[\M(DR^{(1)}_{a-3,b-3,c-2})=\Phi(a-3,b-3,c-2),\]
\[\M(DR^{(1)}_{a-2,b-1,c})=\Phi(a-2,b-1,c),\]
\[\M(DR^{(1)}_{a-1,b-2,c-2})=\Phi(a-1,b-2,c-2),\]
\[\M(DR^{(1)}_{a-1,b-1,c-1})=\Phi(a-1,b-1,c-1),\]
and
\[\M(DR^{(1)}_{a-2,b-2,c-1})=\Phi(a-2,b-2,c-1).\]
By the recurrence (\ref{R2}) and the above five equalities, we obtain
$\M(DR^{(1)}_{a,b,c})=\Phi(a,b,c).$

Similarly, if $c=0$, by Lemmas \ref{dragonlem2}(b) and \ref{lemnew}, the  pair of the numbers of tilings  $(\M(DR^{(1)}_{a,b,c}),$ $ \M(DR^{(2)}_{a,b,c}))$ and the pair of functions $(\Phi(a,b,c),\Psi(a,b,c))$ both satisfy the first equality in the recurrence (\ref{R3}) for $(\bigstar,\lozenge)$. Again, by induction hypothesis, we get (\ref{dram1}).

\medskip

\emph{Case 1.b. $c\geq 2$.}

\medskip

If $\overline{d}\geq 1$, then by Lemmas \ref{dragonlem3} and \ref{lemnew}, $\M(DR^{(1)}_{a,b,c})$ and $\Phi(a,b,c)$ both satisfy the recurrence (\ref{R4}); if $\overline{d}=0$, the later two functions satisfy the same recurrence (\ref{R5}). Then, by induction hypothesis one more time, we get (\ref{dram1}). This means that we have (\ref{dram1}) for any $(a,b,c)\in \overline{\mathcal{D}}_{1}$.

In the same way, if  $(a,b,c)\in \widetilde{\mathcal{D}}_1$, we always get (\ref{dram2}).

\begin{figure}
\centering
\includegraphics[width=8cm]{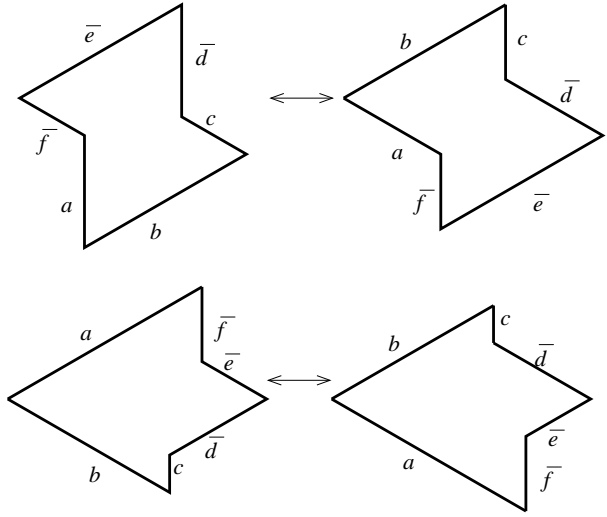}%
\caption{Obtaining a new hexagonal dungeon by flipping the old one.}
\label{fliphex}
\end{figure}

\medskip

\quad \textit{Case 2:} $(a,b,c)\in \overline{\mathcal{D}}_2\cup \widetilde{\mathcal{D}}_2$.

\medskip

We consider first the case when $(a,b,c)\in \overline{\mathcal{D}}_2$.  By flipping the region $DR^{(1)}_{a,b,c}$
 about the $b$-side of the contour $\mathcal{C}^{(1)}(a,b,c)$, we get the region $DR^{(1)}_{\overline{f},\overline{e},\overline{d}}$ (see the  upper row in Figure \ref{fliphex}). Thus,
\[\M(DR^{(1)}_{a,b,c})=\M(DR^{(1)}_{\overline{f},\overline{e},\overline{d}}).\]
On the other hand, by the definition of function $\Phi$, we obtain $\Phi(a,b,c)=\Phi(\overline{f},\overline{e},\overline{d}).$
Therefore, to prove (\ref{dram1}), it suffices to verify that
\begin{equation}\label{dram3}
\M(DR^{(1)}_{\overline{f},\overline{e},\overline{d}})=\Phi(\overline{f},\overline{e},\overline{d}).
\end{equation}

If $\overline{e}\leq 4$, then the triple $(\overline{f},\overline{e},\overline{d})$ are among the ones in the base cases, and (\ref{dram3}) holds. Thus, we can assume that $\overline{e}\geq 5$. We now want to show that $(\overline{f},\overline{e},\overline{d})\in \overline{\mathcal{D}}_1$, then (\ref{dram3}) follows from the Case 1 treated above. It is more convenient to re-write $\mathcal{D}_1$ with all constraints in terms of $a,b,c$ as
\begin{align}
\overline{\mathcal{D}}_1=\{&(a,b,c)\in \mathbb{Z}^3|\, \mathcal{P}^{(1)}(a,b,c)=k, b\geq 5,2b-a-2c+1\geq0,\notag\\
&3b-2a-2a+1\geq 0,2b-a-c+1\geq3, 2\leq a\leq 2b-a-c+1\}.\end{align}
The flips in Figure \ref{fliphex} do not change the perimeter, so the first constraint holds.  Since we assuming $\overline{e}\geq 5$, the second constraint also holds here. Next, we have
$2\overline{e}-\overline{f}-2\overline{d}+1=c\geq 0$,  $3\overline{e}-2\overline{f}-2\overline{d}+1=b\geq 0$, and $2\overline{e}+\overline{f}-\overline{d}+1=c+\overline{d}\geq 3$. Moreover,  $\overline{f}=c+\overline{d}-a+1\geq c+\overline{d}\geq 3$, and  the condition $\overline{f}\leq 2\overline{e}-\overline{f}-\overline{d}+1$ is equivalent to $a\geq 0$, which is obviously true. This means that $(\overline{f},\overline{e},\overline{d})$ is indeed in $\overline{\mathcal{D}}_1$.

Similarly, if $(a,b,c)\in \widetilde{\mathcal{D}}_{2}$, we also flip $DR^{(2)}_{a,b,c}$ over the $b$-side of the contour $\mathcal{C}^{(2)}(a,b,c)$ and get $DR^{(2)}_{\widetilde{f},\widetilde{e},\widetilde{d}}$. Then (\ref{dram2}) follows again from Case 1.

\medskip

\quad \textit{Case 3: } $(a,b,c) \in \overline{\mathcal{D}}_3\cup \widetilde{\mathcal{D}}_{3}$.

\medskip

Assume that $(a,b,c)\in \overline{\mathcal{D}}_3$.


 If $c \geq 2$, then $\M(DR^{(1)}_{a,b,c})$ and $\Phi(a,b,c)$ satisfy the same recurrence (\ref{R1}); and if $c=1$, then the latter functions both satisfy the recurrence (\ref{R2}). Then (\ref{dram1}) follows from the induction hypotheses. We also get (\ref{dram1}) if $c=0$ (in this case the first equality of (\ref{R3}) has been used).
 Thus, (\ref{dram1}) always holds when $(a,b,c)\in \overline{\mathcal{D}}_3$.

 Arguing similarly, we have (\ref{dram2}) holds when $(a,b,c)\in \widetilde{\mathcal{D}}_{3}$. 

\medskip

\textit{Case 4:} $(a,b,c) \in \overline{\mathcal{D}}_4\cup \widetilde{\mathcal{D}}_{4}$.

\medskip

Similar to Case 2, we want to reduce this case to the cases treated before by flipping our region in a suitable way.

If $(a,b,c)\in \overline{\mathcal{D}}_4$ (resp. $\widetilde{\mathcal{D}}_4$),  we get $DR^{(1)}_{b,a,\overline{f}}$ (resp. $DR^{(2)}_{b,a,\widetilde{f}}$)
 by flipping the region $DR^{(2)}_{a,b,c}$ (resp. $DR^{(1)}_{a,b,c}$) about the horizontal line passing the western vertex (see the lower row of Figure \ref{fliphex}). It implies that
\[\M(DR^{(1)}_{a,b,c})=\M(DR^{(2)}_{b,a,\widetilde{f}}) \quad\text{ and }\quad \M(DR^{(2)}_{a,b,c})=\M(DR^{(1)}_{b,a,\overline{f}}).\]
On the other hand, by the definition of the functions $\Phi$ and $\Psi$, we have
\[\Phi(a,b,c)=\Psi(b,a,\widetilde{f})\quad \text{ and } \quad \Psi(a,b,c)=\Phi(b,a,\overline{f}).\]
Therefore, we only need to verify that
\begin{equation}\label{check}
\M(DR^{(1)}_{b,a,\overline{f}})=\Phi(b,a,\overline{f}) \quad \text{ and } \quad \M(DR^{(2)}_{b,a,\widetilde{f}})=\Psi(b,a,\widetilde{f}).
\end{equation}
Let us prove the first equality in (\ref{check}). Similar to Case 2, by the base cases, we can assume $a\geq 7$ and $\overline{e}+\overline{f}\geq 3$. However,  now, we can check easily that $(b,c,\overline{f}) \in \overline{\mathcal{D}}_3$ or $\overline{\mathcal{D}}_1\cup \overline{\mathcal{D}}_2$ if $c=1$ or $0$, respectively. Then the first equality in (\ref{check}) follows from the cases treated before. Do similarly for the second equality in (\ref{check}). This finishes our proof.

\section{More applications of the method.}
In this section, we point out several further applications of the method used in the proof of Theorem \ref{dragonthm}.

Consider the weighted version $DR^{(i)}_{a,b,c}(x)$
 of  the dragon region $DR^{(i)}_{a,b,c}$ by assigning all tiles in Figure \ref{Tileofdragon} a weight 1, and each of other tiles an indeterminate weight $x>0$. Now, $\M(DR^{(i)}_{a,b,c}(x))$ is  the sum of weights of all tilings in $DR^{(i)}_{a,b,c}(x)$, where the weight of a tiling is the product of weights of its constituent tiles.

\begin{figure}\centering
\includegraphics[width=6cm]{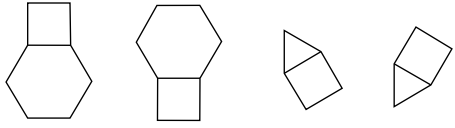}
\caption{All types of tiles having weight 1 in the regions $DR^{(i)}_{a,b,c}(x)$'s.}
\label{Tileofdragon}
\end{figure}

\begin{thm}\label{dragongen}
(a) Assume that $a$, $b$, and $c$ are three non-negative integers satisfying $b\geq 2$, $\overline{d}:=2b-a-2c+1\geq0$ and $\overline{e}:=3b-2a-2c+1\geq0$. Then
\begin{equation}
\M\left(DR^{(1)}_{a,b,c}(x)\right)=2^{\frac{(a-b)(a-b-1)}{2}}(x^2+1)^{A}(x^2+2)^{B}x^{C},
\end{equation}
where\\
$A=(b-c+1)(2b-a-c)+\frac{(b-a)(b-a-1)}{2},$\\
$B=\frac{(a-b+c)(a-b+c-1)}{2},$\\
and
$C=(a-b-1)^2+(a-b+c-1)(a-b+c)-c+\min(2a-2b+c-1,0).$

\medskip

(b) Moreover, if $ \widetilde{d}:=2b-a-2c-1\geq0$ and $ \widetilde{e}:=3b-2a-2c-1\geq0$, then
\begin{equation}
\M\left(DR^{(2)}_{a,b,c}(x)\right)=2^{\frac{(b-a)(b-a-1)}{2}}(x^2+1)^{A'}(x^2+2)^{B'}x^{C'},
\end{equation}
where\\
$A'=(b-c-1)(2b-a-c)+\frac{(a-b)(a-b-1)}{2},$\\
$B'=\frac{(a-b+c)(a-b+c+1)}{2},$\\
and
$C'=(b-a-1)^2+(a-b+c+1)(a-b+c)-\max(2a-2b+c+1,0).$
\end{thm}
We omit the proof of Theorem \ref{dragongen} here, because it is essentially the same as the proof of Theorem \ref{dragonthm}.

Next, we conclude the paper by introducing two new families of regions, which also have the numbers of tilings given by powers of $2$ and $3$.

\begin{figure}\centering
\includegraphics[width=9cm]{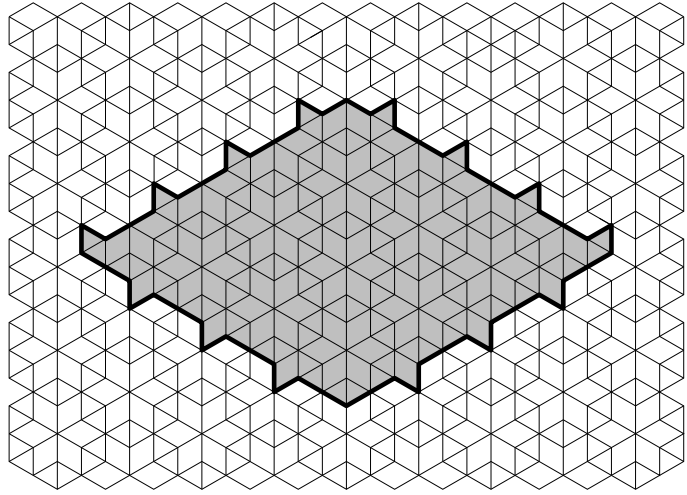}
\caption{The needle rhombus $R_3$.}
\label{newneddle}
\end{figure}

Searching over various families of regions having similar structure to the Aztec diamonds \cite{Elkies1,Elkies2}, we find several ones, in which the tilings are enumerated by  powers of $2$ and $3$. One of these families is the ``needle rhombi" introduced by the author (see Theorem  25 \cite{Tri2}). Figure \ref{newneddle} shows the needle rhombus of order $3$. Inspired by the dragon regions, we generalize the later regions to two new families of six-sided regions as follows.

Partition the triangular lattice into equilateral triangles of side 3, which we call \emph{basic triangles}. Remove all six lattice segments opposite to each vertex of these basic triangles. We get a new lattice, which the needle rhombi live in. Similar to the contours $\mathcal{C}^{(i)}(a,b,c)$, we consider a six-sided contour $\mathcal{C}(a,b,c)$ along the sides of the basic triangles having side-lengths $a,b,c,d:=2b-a-2c, e:=3b-2a-2c,f:=|2b-2a-c|$ (the unit here is the side-length of the basic triangles). Define a two families of  \emph{needle regions} $N^{(1)}_{a,b,c}$ and $N^{(2)}_{a,b,c}$ restricted by the above contour as in Figures \ref{needlenew1}(a) and (b) for the case $a>c+d$, and Figures \ref{needlenew1}(c) and (d) for the case $a\leq c+d$.

\begin{figure}\centering
\includegraphics[width=14cm]{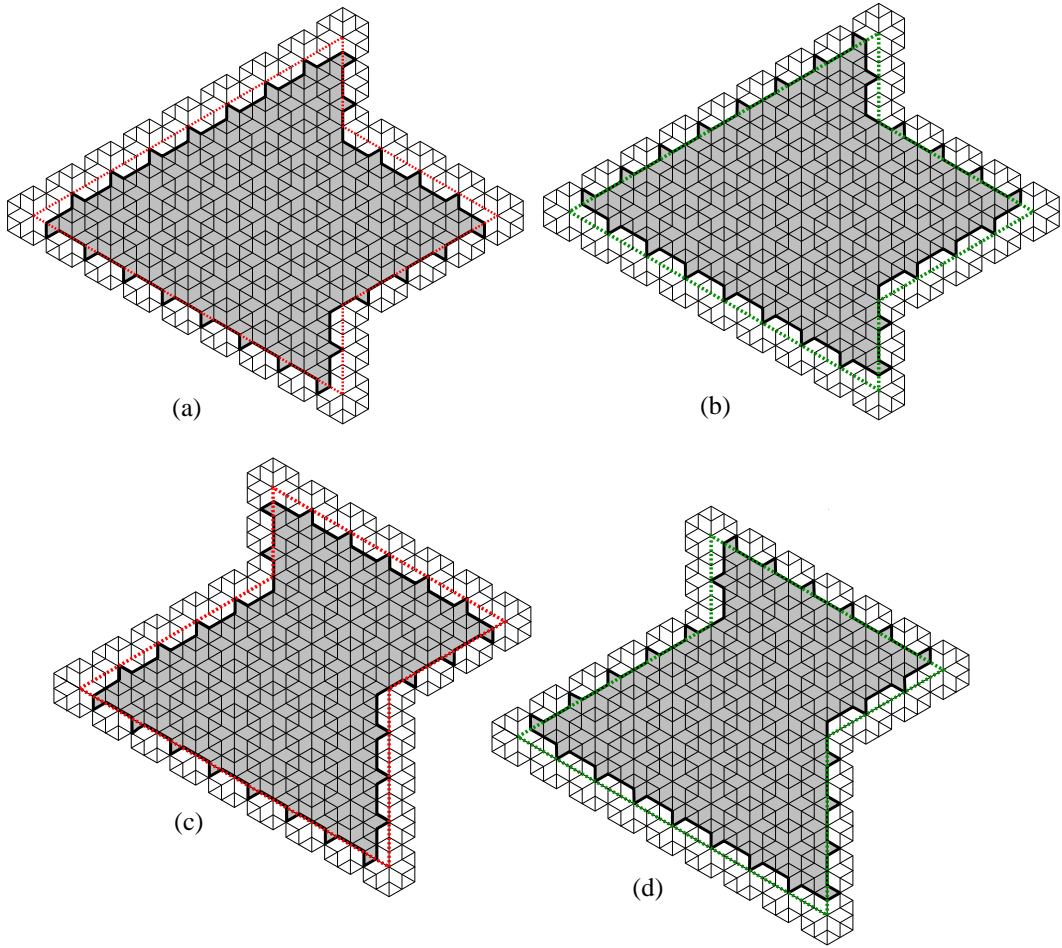}
\caption{The regions (a) $N^{(1)}_{8,8,2}$, (b) $N^{(2)}_{8,8,2}$, (c) $N^{(1)}_{5,8,4}$ and (d) $N^{(2)}_{5,8,4}$.}
\label{needlenew1}
\end{figure}

By repeating the process in enumerating the tilings of the dragon regions, we get the following tiling formulas for the needle regions.

\begin{thm}\label{needlethm}
Assume that $a$, $b$, and $c$ are three non-negative integers satisfying $b\geq 2$, $2b-a-2c\geq0$ and $3b-2a-2c\geq0$. Then

\begin{align}\label{needleeq1}\M(N^{(1)}_{a,b,c})&=2^{a^2-3ab+3b^2-3bc+c^2+ac+a-b+\frac{c}{2}-\frac{1}{2}|2a-2b+c|}\notag\\
&\times3^{(5a-7b+5c+1)(a-b+c)/2+b(b-1)-ac}\end{align}
and
\begin{align}\label{needleeq2}\M(N^{(2)}_{a,b,c})&=2^{a^2-3ab+3b^2-3bc+c^2+ac+a-b+\frac{c}{2}-\frac{1}{2}|2a-2b+c|}\notag\\
&\times3^{(5a-7b+5c+1)(a-b+c)/2+b(b-1)-ac+(b-a)+\min(2a-2b+c,0)}.\end{align}
\end{thm}

The proof of Theorem \ref{needlethm} is completely analogous to that of Theorem \ref{dragonthm}, and will be omitted here.

\end{document}